\documentclass[10pt]{article}
\usepackage{color}
\usepackage{amssymb}
\usepackage{amsthm,array,amssymb,amscd,amsfonts,latexsym, url}
\usepackage{amsmath}
\usepackage[all]{xy}
\newtheorem{theo}{Theorem}[section]
\newtheorem{prop}[theo]{Proposition}

\newtheorem{lemm}[theo]{Lemma}
\newtheorem{coro}[theo]{Corollary}
\newtheorem{rema}[theo]{Remark}
\newtheorem{Defi}[theo]{Definition}

\newtheorem{question}[theo]{Question}

\title{Cycle classes on abelian varieties and the geometry of the Abel-Jacobi map}
\author{Claire Voisin\footnote{The author is supported by the ERC Synergy Grant HyperK (Grant agreement No. 854361).}}
\date{}

\newfont{\gothic}{eufb10}
\begin{document}

\maketitle

\begin{center} {\it Pour  Enrico, avec amiti\'{e}}
\end{center}
\begin{abstract} We discuss two properties of an abelian variety, namely, being a direct summand in a product of Jacobians and the weaker property of being ``split''. We relate the first property to the integral Hodge conjecture for curve classes on abelian varieties. We also relate  both properties to the existence problem for universal zero-cycles on Brauer-Severi varieties over abelian varieties. A similar relation is established for the existence problem of a universal codimension 2 cycle on a cubic threefold.
 \end{abstract}
\section{Introduction}
The purpose of this paper is to explore some geometric  questions related to the integral Hodge conjecture for abelian varieties. This subject  has been revisited recently by Beckmann and de Gaay Fortman   in \cite{degaybeck} who proved the following result (already known by \cite{grabo}  in dimension 3).
\begin{theo}\label{BdGF} Let $A$ be a principally polarized  abelian variety whose minimal class $\gamma_{\rm min}\in H_2(A,\mathbb{Z})$ is algebraic.
Then degree 2  integral Hodge homology classes (or ``curve classes'') of  $A$ are algebraic.
\end{theo}
Here the minimal class is the  integral degree $2$ Hodge homology class  on $A$ defined as follows. Let $g={\rm dim}\,A$ and $\theta\in H^2(A,\mathbb{Z})$ be the class of the principal polarization. Then $\gamma_{\rm min}:=\frac{\theta^{g-1}}{(g-1)!}\in H^{2g-2}(A,\mathbb{Z})\cong H_{2}(A,\mathbb{Z})$.
\begin{rema}{\rm The result proved by Beckman and de Gaay Fortman is in fact more general. In particular, the polarization can be replaced with any line bundle, with no positivity assumption, whose class $c_1(L)\in H^2(A,\mathbb{Z})$ is unimodular.}
\end{rema}
\begin{rema} {\rm As is well-known, the minimal class of the Jacobian of a curve (equipped with its natural polarization) is algebraic, since it is the class of the curve embedded in its Jacobian. Thus Theorem  \ref{BdGF} applies to Jacobians of curves, and also products of Jacobians.}
\end{rema}
We first establish in this paper  the following  complement to Theorem \ref{BdGF}.
\begin{theo} \label{theomain} (Cf. Proposition \ref{theoequiv} and  Corollary \ref{coromain}.) An abelian variety $A$  is a direct summand, as an abelian variety, in a product of Jacobians, if and only if $A\times \widehat{A}$ satisfies the integral Hodge conjecture for curve classes.
\end{theo}
 It follows  that the integral Hodge conjecture for   curve classes on abelian varieties is equivalent to the statement that any abelian variety is a direct summand, as an abelian variety, in a product of Jacobians.
The ``only if'' implication in  this theorem is an immediate consequence of  \cite{degaybeck} (see  Remarks \ref{rema1}, \ref{rema2}).
Note that Proposition \ref{theoequiv} proves a  slightly more precise  statement,  concerning a  principally polarized abelian variety, or more generally an abelian variety which admits a unimodular line bundle.

We will discuss  in this paper  a notion which is weaker than being a direct summand in a product of Jacobians, namely that of a ``split'' abelian variety (see Definition \ref{fisplit}). We first explain our geometric motivation which comes from  the study of  the stable rationality problem for  rationally connected threefolds $X$. It is classically known since the seminal work of Clemens and Griffiths \cite{CG} that the rationality problem  for  rationally connected threefolds $X$ can be solved  by  studying  the intermediate Jacobian  $J:=J^3(X)$ of $X$,  which is a principally polarized  abelian variety.  As discovered in \cite{voisinJAG}, \cite{voisinjems}, the  algebraicity of certain  integral Hodge classes on $J$,  and  integral Hodge classes on the product $J\times X$,   is related to the {\it stable}  rationality of $X$, via
 the geometry of the  Abel-Jacobi map for cycles algebraically equivalent to $0$ on $X$, and we explore  further these phenomena in this paper.  As is well-known, for any smooth projective variety $X$, the group
${\rm Pic}^0(X)={\rm CH}^1(X)_{\rm hom}$ is isomorphic via the Abel-Jacobi  map $\Phi_X^1$ to the intermediate Jacobian $J^1(X)$, and furthermore, there exists a universal divisor
$$\mathcal{P}\in {\rm Pic}(J^1(X)\times X)={\rm CH}^1(J^1(X)\times X)$$ realizing geometrically this isomorphism as $t\mapsto \Phi_X^1(\mathcal{P}_t-\mathcal{P}_0)\in J^1(X)$.
If we now consider  a smooth projective variety $X$ with ${\rm CH}_0(X)=\mathbb{Z}$ and denote  $J:=J^3(X)$  its intermediate Jacobian,   the Abel-Jacobi map for codimension $2$ cycles of $X$  is an isomorphism $\Phi_X: {\rm CH}^2(X)_{\rm alg}\rightarrow J^3(X)$ by \cite{blochsrinivas},  and a universal codimension $2$ cycle is defined in \cite{voisinJAG} to be a cycle $\Gamma\in{\rm CH}^2(J\times X)$, such that the   associated Abel-Jacobi map
$$\Phi_\Gamma: {\rm Alb}(J)=J\rightarrow J$$
is the identity of $J$. Here $\Phi_\Gamma$ (that we will also denote by $\Gamma_*$ in the sequel) is induced by the morphism
$$t\mapsto  \Phi_X(\Gamma_t-\Gamma_0)\in J$$
of  algebraic varieties. An equivalent condition is the fact that
$$[\Gamma]_*: H_1(J,\mathbb{Z})\rightarrow H^3(X,\mathbb{Z})/{\rm torsion}=:H^3(X,\mathbb{Z})_{\rm tf}$$ is the natural isomorphism (recall that, as a complex torus
$$J^3(X)=H^3(X,\mathbb{C})/(F^2H^3(X)\oplus H^3(X,\mathbb{Z})_{\rm tf}),$$
which defines  this natural  isomorphism).
As by definition, $[\Gamma]_*$ is an isomorphism of Hodge structures, it provides a Hodge class in $H^4(J\times X,\mathbb{Z})$ (see \cite[Section 2.2.2]{voisincitrouille}). The existence of a universal codimension $2$ cycle on $J\times X$ is thus a particular instance of the integral Hodge conjecture for degree $4$ Hodge classes on $J\times X$.
As
discovered in \cite{voisininvent},    there are examples  of rationally connected threefolds $X$ which  do not have  a universal codimension $2$ cycle. As observed in \cite{voisinJAG}, the existence of a universal codimension $2$ cycle for a rationally connected smooth projective variety  $X$ is a necessary condition for the existence of a (cohomological) decomposition of the diagonal of $X$, hence for its stable rationality.

 It was however proved in \cite{voisincollino} that  there always exists, for such an $X$ of  dimension $3$, a smooth projective variety $M$ of dimension $d$ (that one can in fact  take to be a surface) with a codimension $2$ cycle $Z\in{\rm CH}^2(M\times X)$ inducing an isomorphism
$$\Phi_Z:{\rm Alb}(M)\cong J^3(X).$$
It thus follows that  the non-existence of a universal codimension $2$ cycle for $X$ implies the non-existence of a universal $0$-cycle for $M$, that is, there does not exist a codimension $d$ cycle
$\Gamma\in {\rm CH}^d(A\times M)$, $A={\rm Alb}(M)$, $d={\rm dim}\,M$, inducing the identity
$$\Gamma_*: {\rm Alb}(A)=A\rightarrow {\rm Alb}(M)=A.$$
For the same reason as above, these examples provide counterexamples to the integral Hodge conjecture on ${\rm Alb}(M)\times M$.
Note that Colliot-Th\'{e}l\`{e}ne \cite{CT} constructed related examples for  varieties  $X$ defined over a non-algebraically closed field.

A natural problem is to try to understand which smooth projective varieties admit a universal $0$-cycle in the above sense.
Trivially, any abelian variety admits a universal $0$-cycle, given by the diagonal. As mentioned above, a curve admits a universal  $0$-cycle, namely its universal divisor.
\begin{question} \label{qintro} (i) (Colliot-Th\'{e}l\`{e}ne \cite{CT}) Let $\psi:P\rightarrow A$ be a  Brauer-Severi variety over an abelian variety. Does $P$ admit a universal $0$-cycle?

(ii) More generally, let  $\psi: P\rightarrow A$ be a fibration into rationally connected varieties. Does $P$ admit a universal $0$-cycle?
\end{question}

\begin{rema}\label{remabrauer} {\rm In both cases, one has ${\rm Alb}(P)\stackrel{\psi_*}{=}{\rm Alb}(A)\cong A$. In  case (i), one might believe that the existence of a universal $0$-cycle $\Gamma\in{\rm CH}^d(A\times P)$ inducing the identity $\Gamma_*:{\rm Alb}(A)=A\rightarrow {\rm Alb}(P)=A$ forces the Brauer class to be $0$, but this is not the case, since the condition on $\Gamma$ concerns only its action on homology of degree $1$ of $A$ and $P$, namely,  $ \psi_*\circ  \Gamma_* $ must be the identity of  $H_1(A,\mathbb{Z})$. It does not say that the Brauer class of $P$ vanishes because the action of $ \psi_*\circ  \Gamma_* $  on the higher degree homology groups of $A$  (especially the group $H_{2g}(A,\mathbb{Z})$ which controls the index of the fibration $\psi$, hence  the Brauer class) is  not specified.}
\end{rema}
Question \ref{qintro} can be considered as a particular case of the integral Hodge conjecture for  Brauer-Severi varieties over abelian varieties. Indeed, as mentioned above, our problem can be formulated as follows:
\begin{question} \label{questionunivreform} Let $\psi:P\rightarrow A$ be a  Brauer-Severi variety over an abelian variety. Does there exist a cycle $\Gamma\in {\rm CH}^d(A\times P)$, $d={\rm dim}\,P$, such that $[\Gamma]_*:H_1(A,\mathbb{Z})\rightarrow H_1(P,\mathbb{Z})$ is the inverse of $\psi_*$?
\end{question}
The cycle $\Gamma$ is a cycle on the  Brauer-Severi variety $A\times P$ over $A\times A$ and the last condition is a restriction on the K\"{u}nneth component of type $(1,2d-1)$ of $[\Gamma]$.
Note that the integral Hodge conjecture for  Brauer-Severi varieties $P\rightarrow A$ over abelian varieties satisfying the Hodge conjecture  has a negative answer in general by \cite{hotchkiss}. We are going to relate here  the problem above to the integral Hodge conjecture on $A\times A$. To start with,
an easy  general result on Question \ref{qintro}(ii) is the following
\begin{prop} \label{leGHS} Let $A$ be an abelian variety and let $P\rightarrow A$ be a fibration with rationally connected general fiber. Assume that $A$ is  a direct summand in a product of Jacobians. Then $P$ admits a universal $0$-cycle.
\end{prop}
This proposition applied to the case where $P$ is a Brauer-Severi variety provides many examples where there is a universal $0$-cycle while the Brauer class is nontrivial (see Remark \ref{remabrauer}). Furthermore, by Theorem \ref{theomain}, we conclude that the integral Hodge conjecture for curve classes on abelian varieties implies a positive answer to Question \ref{qintro}(i). Our next result is a weak converse to Proposition \ref{leGHS} for which we introduce the following
\begin{Defi}\label{fisplit} An abelian variety of dimension $g$  is said to be split if there exists a codimension $g$ cycle $\Gamma\in{\rm CH}^g(A\times A)$ such that  the class
$[\Gamma]\in H^{2g}(A\times A,\mathbb{Z})$ acts on $H_*(A,\mathbb{Z})$ as  the      K\"{u}nneth projector $\delta_1$ onto  $H_1(A,\mathbb{Z})$.
\end{Defi}
\begin{rema}{\rm The K\"unneth projectors $\delta_i$ being integral Hodge classes on $A\times A$, they are algebraic if $A\times A$ satisfies the integral Hodge conjecture.  }
\end{rema}

We will establish a few general facts in Section \ref{secgen1split}. In particular we will prove Proposition \ref{ledejaprouve} which says that  splitness of abelian varieties is implied by  the integral Hodge conjecture for curve classes on abelian varieties.

To state our next result, let us say that an abelian variety is   ``Mumford-Tate general'' if $\rho(A)=1$ and the Mumford-Tate group of the Hodge structure on $H^1(A,\mathbb{Q})$ is the symplectic group of the skew-pairing given by the polarization. This assumption is satisfied by the polarized abelian variety parameterized by a very general point in the  moduli space of polarized abelian varieties with polarization of a given  type, but it is in fact a more precise statement, which is satisfied for example by a very general Jacobian of curve or intermediate Jacobian of  a cubic threefold. When $\rho(A)=1$, the group of integral  Hodge classes of degree $2g-2$ is also cyclic generated by a class $\gamma_{\rm min}$.
We will prove in Section \ref{secunibra} the following
\begin{theo}\label{theobrauer0univ}  Let $A$ be a Mumford-Tate general abelian  variety with $\rho(A)=1$.  Assume that the intersection number $c_1(L)\cdot \gamma_{\rm min}$ is even.  Then, if  any Brauer-Severi variety $P\rightarrow A$  admits a universal $0$-cycle,  $A$ is split.
\end{theo}

\begin{rema} {\rm If $A$ is principally polarized, the condition that $c_1(L)\cdot \gamma_{\rm min}$ is even is equivalent to the dimension of $A$   being even. If the polarization is of type $(1,\ldots,1,d)$, this condition says that $d{\rm dim}\,A$ is even.}
\end{rema}
To summarize our results,  for  a Mumford-Tate general even dimensional  abelian variety $A$ with Picard number $1$, we have  implications as follows:
   $A$ being  a direct summand in a product of Jacobians implies  a positive answer to Question \ref{qintro}(i) and (ii), and in the other direction, a positive answer to Question \ref{qintro}(i) for any $P$ implies that $A$ is  split. The question whether ``split'' implies   that  $A$ is  a direct summand in a product of Jacobians (so that the three statements are equivalent) remains open.

Question \ref{qintro}(ii)  is motivated by the stable rationality problem for the cubic threefold, and more specifically, by the following
\begin{question}\label{questioncubic} Let $X$ be a smooth cubic threefold. Does $X$ admit a universal codimension $2$ cycle?
\end{question}

Questions \ref{qintro} and \ref{questioncubic} are directly linked  by the Iliev-Markushevich-Tikhomirov construction \cite{markuti}, \cite{ilievmarku} which describes a Zariski open set of the Hilbert scheme  $H_{5,1}$ of elliptic curves of degree $5$ in a cubic threefold $X$ as  a fibration into $\mathbb{P}^5$ over a Zariski open set of the intermediate Jacobian $J=J^3(X)$ of $X$. More precisely, we will prove
\begin{prop} \label{procubiceasy} Let $X$ be a smooth  cubic threefold and let $\widetilde{H}_{5,1}$ be a smooth projective model of $H_{5,1}$. Then
$X$ admits a universal codimension $2$ cycle if and only if $\widetilde{H}_{5,1}$ admits a universal $0$-cycle.
\end{prop}
 In the case of the fibration $\widetilde{H}_{5,1}\rightarrow J$, the argument  leading to the proof of Proposition  \ref{leGHS} had been used  in  \cite{voisinJAG}  to prove
\begin{theo}\label{theovoisinJAG} A smooth  cubic threefold admits a universal codimension $2$ cycle if the minimal class of its intermediate Jacobian is algebraic.\end{theo}

Our last result  is   a partial converse to  Theorem \ref{theovoisinJAG},  which will be proved in Section \ref{sectheocubicsplit}. We will say that a cubic threefold is Mumford-Tate general if its intermediate Jacobian $J$  is.
\begin{theo}\label{theocubicsplit} Let $X$ be a Mumford-Tate general cubic threefold. Then  if $X$ admits a universal codimension $2$ cycle, $J$ is split.
\end{theo}

{\bf Thanks.} {\it I thank Jean-Louis Colliot-Th\'{e}l\`{e}ne and James Hotchkiss  for inspiring discussions and correspondence. I also thank  Olivier de Gaay Fortman for interesting exchanges and the referee for his/her careful reading.}
\section{Curve classes on abelian varieties}
We establish  in this section Theorem \ref{theomain}. In the case of abelian varieties equipped with a unimodular line bundle, the following stronger statement holds.

\begin{prop}\label{theoequiv} Let $A$ be a $g$-dimensional  abelian variety equipped with  a line bundle $L$ such that $L^g=\pm g!$ (that is, $L$ is unimodular). Then if $A$ satisfies the integral Hodge conjecture for degree $2$ integral Hodge homology  classes (or ``curve classes''), it is a direct summand  in a product of Jacobians.
\end{prop}
\begin{rema}{\rm  We do not ask that $L$ is ample in Proposition \ref{theoequiv}, so $A$ is not necessarily a principally polarized abelian variety.   }
\end{rema}
\begin{rema} {\rm Even if $L$ is ample, that is, a principal polarization, we just ask that $A$ is a direct summand as an abelian variety, and not that the natural polarization on the product of Jacobians restricts to $L$. Indeed, the last condition is much too strong by the usual Clemens-Griffiths argument: if $A$ is simple, this would imply that $A$ is isomorphic to the Jacobian of a curve.}
\end{rema}
\begin{rema}\label{rema1} {\rm It is proved in \cite{degaybeck} (see also Theorem \ref{BdGF}),  that a product of Jacobians satisfies the integral Hodge conjecture for curve classes. If $j:A\hookrightarrow J$ is the inclusion of a direct summand in such a product $J$, then for any Hodge class $\alpha$ on $A$, $j_*\alpha$ is an integral Hodge class on $J$, and $\alpha$ is algebraic on $A$ if and only if it is algebraic on $J$, using a left inverse  $\pi:J\rightarrow A$ of $j$. Hence the implication in Proposition \ref{theoequiv} is in fact an equivalence.}
\end{rema}

\begin{coro}\label{coromain} (Cf. Theorem \ref{theomain}.) Let $A$ be any abelian variety. Then the integral Hodge conjecture for curve classes on  $A\times \widehat{A}$  implies that $A$  is  a direct summand in a product of Jacobians.

The integral Hodge conjecture for curve classes on abelian varieties thus implies that any abelian variety is  a direct summand in a product of Jacobians.
\end{coro}
\begin{proof}
 If $A$ is an abelian variety, $A\times \widehat{A}$ admits a line bundle $L$ as in Proposition \ref{theoequiv}, namely the Poincar\'{e} divisor (this  is the  starting point in  the Zarhin trick). As $A$ is a direct summand in $A\times \widehat{A}$, it is a direct summand in a product of Jacobians if so is $A\times \widehat{A}$.  Proposition \ref{theoequiv} applies to $A\times \widehat{A}$ and thus the integral Hodge conjecture for curve classes on $A\times \widehat{A}$ implies that $A\times \widehat{A}$ is a direct summand in a product of Jacobians.
\end{proof}
\begin{rema}\label{rema2} {\rm  Again, the implication of the corollary is in fact an equivalence by \cite{degaybeck}, since if $A$ is a direct summand in a product of Jacobians, then $A\times \widehat{A}$ is also a direct summand in a product of Jacobians, hence satisfies the integral Hodge conjecture for curve classes as  explained in Remark \ref{rema1}}.\end{rema}
\begin{proof}[Proof of Proposition \ref{theoequiv}] We consider the minimal class $\gamma_{\rm min}:=\frac{c_1(L)^{g-1}}{(g-1)!}$. By assumption, it is algebraic on $A$, hence there exist smooth projective curves $C_i$, and morphisms $j_i:C_i\rightarrow A$, such that
\begin{eqnarray}\label{eqclass} \gamma_{\rm min}=\sum_i\epsilon_i j_{i*}[C_i]_{\rm fund}\,\,{\rm in}\,\,H_2(A,\mathbb{Z}),
\end{eqnarray}
where $\epsilon_i=\pm1$, and where we used the complex orientation of $A$ to define  the Poincar\'{e} duality isomorphism $H^{2g-2}(A,\mathbb{Z})\cong H_2(A,\mathbb{Z})$.
Using $L$, the dual abelian variety $\widehat{A}:={\rm Pic}^0(A)$ is isomorphic to $A$. We have a natural morphism
\begin{eqnarray}\label{eqpullback}  \hat{j}: A\cong \widehat{A}\stackrel{(\hat{j}_i)}{\rightarrow} \prod_i{\rm Pic}^0(C_i)=\prod_iJ(C_i)
\end{eqnarray}
which is induced by the pull-back maps $\hat{j}_i:=j_i^*:{\rm Pic}^0(A)\rightarrow {\rm Pic}^0(C_i)$.
The abelian variety $ J:=\prod_iJ(C_i)$ admits the divisor  $\Theta_{\epsilon_\cdot}$ defined as
\begin{eqnarray}\label{eqthetaepsilon} \Theta_{\epsilon_\cdot}:=\sum_i\epsilon_i{\rm pr}_i^*\Theta_i,
\end{eqnarray}
where ${\rm pr}_i: J\rightarrow J(C_i)$ is the natural projection and $\Theta_i$ is the natural Theta-divisor on $J(C_i)$. The classes  $[\Theta_{\epsilon_\cdot}]\in H^2(J,\mathbb{Z})$, resp.   $c_1(L)\in H^2(A,\mathbb{Z})$  provide equivalently skew-symmetric  intersection pairings $\langle\,,\,\rangle_J$ on $H_1(J,\mathbb{Z})$, resp.  $\langle\,,\,\rangle_L$ on  $H_1(A,\mathbb{Z})$.
We now have
\begin{lemm} \label{lecomputcoh} The restriction $\hat{j}^*\Theta_{\epsilon_\cdot}$ is cohomologous to $\epsilon c_1(L)$, where $\epsilon=\pm1$ is the sign of $c_1(L)^g$. Equivalently, the restriction of $\langle\,,\,\rangle_J$ to $\hat{j}_*H_1(A,\mathbb{Z})$ is equal to $\epsilon\langle\,,\,\rangle_L$.
\end{lemm}
We postpone the proof of the lemma and conclude the proof of the proposition.
 As $\frac{L^g}{g!}=\pm 1$, the intersection pairing $\langle\,,\,\rangle_A$ is unimodular. By Lemma \ref{lecomputcoh}, the restriction of $\langle\,,\,\rangle_J$ to  $\hat{j}_*H_1(A,\mathbb{Z})$ is unimodular.  It follows  that there is a direct sum decomposition
\begin{eqnarray} \label{eqdecompp} H_1(J,\mathbb{Z})=H_1(A,\mathbb{Z})\oplus H_1(A,\mathbb{Z})^{\perp},
\end{eqnarray}
which is orthogonal with respect to $\langle\,,\,\rangle_J$. As the pairing is induced by a $(1,1)$-class, it satisfies the Hodge-Riemann relations and thus the orthogonal decomposition (\ref{eqdecompp}) is compatible with the Hodge decomposition, hence induces a direct sum decomposition
$$J=A\oplus A^{\perp}.$$
\end{proof}
\begin{proof}[Proof of Lemma \ref{lecomputcoh}] The class $c_1(L)$,  or equivalently the pairing $\langle\,,\,\rangle_A$, induces an isomorphism
$\iota_L: H^1(A,\mathbb{Z})\cong H_1(A,\mathbb{Z})$. We claim that for any $\alpha,\,\beta \in H^1(A,\mathbb{Z})$,
 \begin{eqnarray}\label{eqintgammamin} \int_A\gamma_{\rm min}\cup \alpha\cup \beta=\langle \iota_L(\alpha),\iota_L(\beta)\rangle_L,
 \end{eqnarray}
 where the orientation of $A$ is chosen in such a way that $\int_A c_1(L)^g=g!$.
 This formula is standard (see \cite{debarrebook}) in the context of principally polarized abelian varieties. It  is proved using a basis $e_1,\ldots,e_{2g}$ of $H^1(A,\mathbb{Z})$ for which
 $c_1(L)=\sum_{i=1}^ge_i\wedge e_{i+g}$ in $ \wedge^2H^1(A,\mathbb{Z})$. Then the isomorphism $\iota_L$ maps $e_i$ to $e_{i+g}^*$ and $e_{i+g}$ to $-e_{i}$ for $i\leq g$. Furthermore for $i,\,j\leq g$
 \begin{eqnarray} \label{eqsumgammaminprodL}\langle e_i^*,e_{g+j}^*\rangle_L=\delta_{ij}\end{eqnarray}
 while $\langle e_i^*,e_{j}^*\rangle_L=0$ for $i\leq g,\,j\leq g$.
 By definition,
 $$ \gamma_{\rm min} =\sum_i e_1\wedge e_{1+g}\wedge \ldots \widehat{e_i\wedge e_{i+g}}\ldots\wedge e_g\wedge e_{2g}\,\,{\rm in}\,\,\bigwedge^{2g-2}H^1(A,\mathbb{Z})\cong H^{2g-2}(A,\mathbb{Z})$$
 so that one gets for any $i,\,j$
 \begin{eqnarray}\label{eqsumgammamin} \int_A\gamma_{\rm min}\cup e_i\cup e_j=0\,\,{\rm for}\,i<j,\,j\not=i+g\\
 \nonumber
 \int_A\gamma_{\rm min}\cup e_i\cup e_{i+g}=1.
 \end{eqnarray}
 Comparing (\ref{eqsumgammaminprodL}) and (\ref{eqsumgammamin}) gives the result.

 If we see as in (\ref{eqclass}) $\gamma_{\rm min}$ as a degree $2$ homology class (rather than a degree $2g-2$  cohomology class) on $A$, (\ref{eqintgammamin}) rewrites as
 \begin{eqnarray}\label{eqintgammamin1} \int_{\gamma_{\rm min}} \alpha\cup \beta=\epsilon \langle \iota_L(\alpha),\iota_L(\beta)\rangle_L,
 \end{eqnarray}
 since the complex orientation of $A$ and the orientation used in (\ref{eqintgammamin}) differ by the sign $\epsilon$.
Using (\ref{eqclass}), we rewrite (\ref{eqintgammamin1}) as
 \begin{eqnarray}\label{eqintgammaminpullback}\langle \iota_L(\alpha),\iota_L(\beta)\rangle_L=\epsilon\sum_{i}\epsilon_i\int_{C_i}j_i^*\alpha\cup j_i^*\beta.
 \end{eqnarray}
 We apply in turn (\ref{eqintgammamin1}) to each Jacobian $J(C_i)$ equipped with its principal polarization $\Theta_i$ and minimal class $[C_i]$ and get
 \begin{eqnarray}\label{eqintgammaminpullbackmieux}\langle \iota_L(\alpha),\iota_L(\beta)\rangle_L=\epsilon\sum_{i}\epsilon_i\langle \iota_{\Theta_i}(j_i^*\alpha), \iota_{\Theta_i}(j_i^*\beta)\rangle_{\Theta_i}.
 \end{eqnarray}
 This concludes the proof of Lemma \ref{lecomputcoh} by general duality, recalling that we are looking at the dual embedding
  $$A\cong \widehat{A}\stackrel{\hat{j}=(\hat{j_i})}{\rightarrow} \prod J(C_i),$$
   where  $\hat{j_i}=j_i^*: \widehat{A}\rightarrow J(C_i)$.  One observes that, by definition of $\hat{j}_i:A\rightarrow J(C_i)$, for any $\alpha\in H^1(A,\mathbb{Z})$, \begin{eqnarray}\label{eqgeneralduality} \iota_{\Theta_i}(j_i^*\alpha)=\hat{j_{i}}_*(\iota_L(\alpha))\,\,{\rm in}\,\,H_1(J(C_i),\mathbb{Z}).
   \end{eqnarray}
   This allows to rewrite the right hand side of (\ref{eqintgammaminpullbackmieux}) as follows
   \begin{eqnarray}\epsilon\label{eqfinfin}\sum_{i}\epsilon_i\langle \iota_{\Theta_i}(j_i^*\alpha), \iota_{\Theta_i}(j_i^*\beta)\rangle_{\Theta_i}=\epsilon
   \sum_{i}\epsilon_i\langle \hat{j_{i}}_*(\iota_L(\alpha)), \hat{j_{i}}_*(\iota_L(\beta)))\rangle_{\Theta_i}\\
   \nonumber
   =\epsilon\langle \hat{j}_*(\iota_L(\alpha)),\hat{j}_*(\iota_L(\beta))\rangle_{J},
   \end{eqnarray}
   which concludes the proof by (\ref{eqintgammaminpullbackmieux}).
\end{proof}

\subsection{Split abelian varieties\label{secgen1split}}
The algebraicity with $\mathbb{Q}$-coefficients of the K\"{u}nneth projectors of any $g$-dimensional  abelian variety $A$ is well-known (see \cite{lieberman}) and  follows from the fact that $A\times A$ contains codimension $g$ subvarieties $\Gamma_i$ defined as the graph of multiplication by $i$ for any integer $i$. It also contains their transpose $^t\Gamma_i$. As $[\Gamma_i]_*$ acts by multiplication by $i^k$ on $H_k(A\times A,\mathbb{Z})$, one gets
\begin{eqnarray}\label{eqgraphmult} [\Gamma_i]=\sum_{k=0}^{2g} i^k\delta_k,\,\\
\nonumber
[^t\Gamma_i]=\sum_{k=0}^{2g} i^{2g-k}\delta_k.
\end{eqnarray}
These equations imply that the $\delta_i$ are algebraic with $\mathbb{Q}$-coefficients, as shows the nonvanishing of a Vandermonde determinant (for adequate choices of integers $i_0,\ldots,\,i_{2g}$), which will appear in the denominator. If we want to analyze the situation with $\mathbb{Z}$-coefficients, (forgetting about the polarization which can be of very large degree and not bring any further information), we argue as follows.
Equations (\ref{eqgraphmult}) show that, in the sublattice $L$ of $H^{2g}(A\times A,\mathbb{Z})$ generated by the $\delta_k$,
the group of algebraic classes contains
$\sum_{k=0}^{2g} i^k\delta_k$ and $\sum_{k=0}^{2g} i^{2g-k}\delta_k$ for any $i$. It seems  possible that (if the polarization has sufficiently divisible degree)  no other combination of the $\delta_k$ is algebraic on $A\times A$. The subgroup $L'\subset L$ generated by $\sum_{k=0}^{2g} i^k\delta_k$ and $\sum_{k=0}^{2g} i^{2g-k}\delta_k$ for any $i$ is not the whole group generated by the $\delta_k$. Indeed, consider the dual lattice $L^*$. An element of $L^*$ is a combination
$$P=\sum_{k=0}^{2g} \alpha_k\delta_k^*$$
and we associate to $P$ the polynomial $P(x)=\sum_{k=0}^{2g} \alpha_kx^k$.
We now consider the group $M\cong (L')^*$ of elements $P\in L^*\otimes \mathbb{Q}$ that restrict to elements of $(L')^*\subset (L')^*\otimes \mathbb{Q}$. If a K\"{u}nneth projector $\delta_k$ satisfies $\mu\delta_k\in L'$, one has $$\mu\alpha_k\in\mathbb{Z}$$ for any $P=\sum_{k=0}^{2g} \alpha_k\delta_k^*\in M$, so we need to know what are the denominators of elements $P\in M$. Such an element $P$ satisfies by (\ref{eqgraphmult}) the conditions

\begin{eqnarray}\label{eqpoly1} \sum_{k=0}^{2g} \alpha_ki^k\in\mathbb{Z}\,\,\forall i\in \mathbb{N},\\
\nonumber
\sum_{k=0}^{2g}\alpha_k i^{2g-k} \,\,\forall i\in \mathbb{N}.
\end{eqnarray}
The corresponding polynomial $P(x)$ thus has the property that
\begin{eqnarray}\label{eqpoly2}P(i)\in \mathbb{Z}\,\,\forall i\in \mathbb{N},\, ^tP(i)\in \mathbb{Z}\,\,\forall i\in \mathbb{N},\end{eqnarray}
where $^tP$ is the reciprocal polynomial of $P$. Polynomials $P(x)=\sum_{k=0}^{2g} \alpha_kx^k$ with rational coefficients taking integral values on integers are well-known to be combinations with integral coefficients of binomial polynomials. We also have to take into account the reciprocal condition to compute the denominators of the elements in $M$. In dimension $2$, we get that $2\delta_i$ is algebraic for any $i$.

Recall from Definition \ref{fisplit} that an abelian variety $A$ is split if its first K\"{u}nneth projector $\delta_1$ on $H_1(A,\mathbb{Z})$ is algebraic.
This property is related to Theorem \ref{theomain} by  the following
\begin{prop} \label{ledejaprouve} Let $A$ be an abelian variety.

(i) Assume $A$ is principally polarized (or has a unimodular line bundle) and the minimal class $\gamma_{\rm min}$ is algebraic. Then $A$ is  split.

(ii) Assume $A$ is a direct summand in a product of Jacobians, then $A$ is split.

(iii)  Assume the integral Hodge conjecture holds for curve classes on abelian varieties. Then any abelian variety is split.
\end{prop}
We will use the following
\begin{lemm}\label{exintro}  Let $A=J(C)$ be the Jacobian of a curve. Then $A$ is split.
\end{lemm}
\begin{proof} Indeed, let $j: C\hookrightarrow A$ be the canonical morphism determined by a $0$-cycle of degree $1$ on $C$.  Recalling that $A$ is isomorphic to its dual $\widehat{A}={\rm Pic}^0(A)$, denote by $\mathcal{P}$ the Poincar\'{e} divisor on $A\times A$. We consider the restriction of $\mathcal{P}$ to $A\times C$, that we denote by $\mathcal{P}_C\in{\rm CH}^1(A\times C)$. As $[\mathcal{P}]_*$ acts as the Poincar\'{e} duality isomorphism $H_1(A,\mathbb{Z})\cong H^1(A,\mathbb{Z})$ and trivially on other cohomology groups,   $[\mathcal{P}_C]_*$ induces  an isomorphism $H_1(A,\mathbb{Z})\cong H^1(C,\mathbb{Z})$, since the restriction map $j^*: H^1(A,\mathbb{Z})\rightarrow H^1(C,\mathbb{Z})$ is an isomorphism. Furthermore  $[\mathcal{P}_C]_*$  acts trivially on the other cohomology groups.  As $j_*: H_1(C,\mathbb{Z})\rightarrow H_1(A,\mathbb{Z})$ is an isomorphism, the class of the $g$-cycle $(Id_A,j)_*\mathcal{P}_C\in {\rm CH}_g(A\times A)$ acts as the identity of  $H_1(A,\mathbb{Z})$ and trivially on the other cohomology groups.
\end{proof}

\begin{proof}[Proof of Proposition \ref{ledejaprouve}]  In both cases (i) and (iii), it follows from Proposition \ref{theoequiv} and Corollary \ref{coromain} that $A$ is a direct summand in a product $J$ of Jacobians, so we only have to prove (ii). Let $g:={\rm dim}\,A$ and $g':={\rm dim}\,J$. By Lemma \ref{exintro}, $J$ is split, hence there is  a cycle $\Gamma\in {\rm CH}^{g'}(J\times J)$ such that
$[\Gamma]$ acts as the projector onto $H_1(J,\mathbb{Z})$. Let $j:A\rightarrow J$ be the inclusion and $\pi:J\rightarrow A$ be the projection. Then
if $$\Gamma':=(Id_A,\pi)_*(j,Id_J)^*\Gamma\in {\rm CH}^g(A\times A),$$
we have
$$[\Gamma']_*=\pi_*\circ j_*: H_*(A,\mathbb{Z})\rightarrow H_*(A,\mathbb{Z}),$$
hence $[\Gamma']$ acts as the projector onto $H_1(A,\mathbb{Z})$ and $A$ is split.
\end{proof}
\begin{rema}{\rm If $A$ is  a direct summand in a Jacobian, one can easily prove by the same arguments as above as above  that all K\"{u}nneth projectors $\delta_i:H_*(A,\mathbb{Z})\rightarrow H_i(A,\mathbb{Z})\hookrightarrow H_*(A,\mathbb{Z})$ are algebraic. It is not so clear however that this last property  holds if we only assume that $A$ is split.}
\end{rema}
\section{Existence of universal $0$-cycles \label{secisplit}}
\subsection{Universal $0$-cycle on rationally connected fibrations over abelian varieties \label{seceasyprop}}

 We give in this section  the proof of   Propositions \ref{leGHS} and \ref{procubiceasy}.
\begin{proof}[Proof of Proposition \ref{leGHS}] Let $\psi:P\rightarrow A$ be a fibration with rationally connected general fiber. Assume that $A$ is a direct summand in a direct sum $J=\prod_{i=1}^kJ(C_i)$ of Jacobians and denote respectively  by $j:A\rightarrow J$ and $\pi:J\rightarrow A$ the inclusion and the projection. We consider the fibered product $\psi_J:P_J:=J\times_A P\rightarrow J$ and observe that the Graber-Harris-Starr theorem \cite{GHS} applies to the restriction of $\psi_J$ over any curve passing through the  general  point of  $J$. Thus for the general translate of  any curve $C\subset J$, there is a lift $\sigma: C\rightarrow  P_J$ of $C$ in $P_J$, with graph $\Gamma_{\sigma}\in {\rm CH}^n(C\times P_J)$, where $n={\rm dim}\,P_J$. By assumption, $J=\prod_iJ(C_i)$ where $C_i$ is a smooth curve of genus $g_i$. We thus know that there is an embedding $C_i\subset J$ for each $i$ and a cycle $\Gamma_{\sigma_i}\in {\rm CH}^n(C\times P_J)$ as above.
 We can thus construct  a cycle
$$Z\in{\rm CH}^n(C_1^{(g_1)}\times\ldots \times C_k^{(g_k)}\times P_J)$$
defined as
$$Z=\sum_i{\rm pr}_i^*Z_i,\,$$ where $Z_i\in {\rm CH}^n(C_i^{(g_i)}\times P_J)$ is the cycle whose pull-back to $C^{g_i}\times P_J$ is the symmetric cycle $\sum_{j=1}^{g_i} p_j^*\Gamma_{\sigma_i}$ (the $p_j$ being the projections from $C_i^{g_i}$ to $C_i$). Furthermore, we have a birational map
 $\tau: \prod_iC_i^{(g_i)}\rightarrow J$ such that ${\rm alb}_{C_i}\circ {\rm pr}_i\circ \tau^{-1}$ is the projection from $J$ to $J(C_i)$ and we now set
$$Z':=(\tau,Id_P)_*Z\in {\rm CH}^n(J\times P).$$ For each $i$, the cycle $\Gamma_{\sigma_i}$ has the property that $$\Gamma_{\sigma_i*}:{\rm Alb}(C_i) \rightarrow {\rm Alb}(P_J)=J$$
is the inclusion of $J(C_i)$ in $J$, hence the cycle
$Z$
has the property that
$$ Z_*:{\rm Alb}(\prod_iC_i^{(g_i)})=J\rightarrow {\rm Alb}(P_J)=J$$
is the identity. It follows that the cycle $Z'$ satisfies as well  the property that
$$Z'_*:{\rm Alb}(J)=J\rightarrow {\rm Alb}(P_J)=J$$
is the identity so $Z'$ is a universal $0$-cycle for $P_J$.
Finally, let $\pi_P:P_J\rightarrow P$ be the natural projection, and let
$$Z'':=\pi_P\circ Z'\circ j\in {\rm CH}^m(A\times P),\,\,m={\rm dim}\,P.$$
It is clear that $Z''$ is a universal $0$-cycle for $P$.
\end{proof}
Using  Proposition \ref{theoequiv}, we deduce
\begin{coro} (Cf. \cite{voisinJAG}) Let  $\psi:P\rightarrow A$ be a fibration with rationally connected general fiber. If $A$ has a unimodular line bundle such that the minimal class $\gamma_{\rm min}$ is algebraic, $P$ admits a universal zero-cycle.
\end{coro}

We now turn to the case of the Iliev-Markushevich-Tikhomirov fibration \cite{ilievmarku}, \cite{markuti}. As in the introduction, $X$ is a smooth cubic threefold, and $\widetilde{H}^{5,1}$ is a smooth projective model of the Hilbert scheme $H^{5,1}$ of degree 5, genus 1, curves in $X$.
\begin{proof}[Proof of Proposition \ref{procubiceasy}] We have ${\rm Alb}(\widetilde{H}_{5,1})=J^3(X)=:J$ since $\widetilde{H}_{5,1}$ is fibered over $J^3(X)$ (via the Abel-Jacobi map) into rationally connected $5$-folds. Assume that $\widetilde{H}_{5,1}$ admits a universal $0$-cycle $\Gamma\in{\rm CH}^{10}(J\times \widetilde{H}_{5,1})$. The Zariski closure $\mathcal{E}\subset \widetilde{H}_{5,1}\times X$ of the universal elliptic curve of degree $5$
 gives a codimension $2$ cycle in $\widetilde{H}_{5,1}\times X$, which  induces
the  Iliev-Markushevich-Tikhomirov Abel-Jacobi isomorphism (see \cite{ilievmarku}, \cite{markuti})
$$\mathcal{E}_*:{\rm Alb}(\widetilde{H}_{5,1})\cong J^3(X)=J.$$
 Consider the composition
 $$\Gamma_X:=\mathcal{E}\circ \Gamma\in {\rm CH}^2(J\times X).$$
 Then
 $$\Gamma_{X*}=\mathcal{E}_*\circ \Gamma_*:{\rm Alb}(J)=J\rightarrow J^3(X)=J$$
 is the identity of $J$, hence $X$ has a universal codimension $2$ cycle.
In the other direction, suppose that $X$ has  a universal codimension $2$ cycle
$\Gamma_X\in{\rm CH}^2(J\times X)$. We claim that there exists a cycle
$\Gamma\in {\rm CH}^{10}(J\times \widetilde{H}_{5,1})$ such that
$$\Gamma_{X*}=\mathcal{E}_*\circ \Gamma_*:J\rightarrow J.$$
The claim immediately implies that $\Gamma_*$ is a universal $0$-cycle for $\widetilde{H}_{5,1}$ since
\begin{eqnarray}\label{eqisoajell} \mathcal{E}_*:{\rm Alb}(\widetilde{H}_{5,1})\rightarrow J^3(X)=J
\end{eqnarray}  is an isomorphism.
To prove the claim, we use the Shen universal generation theorem \cite{mshen}, which says the following. Denoting by  $\Sigma$ the surface of lines in $X$, and by $P_\Sigma\subset \Sigma\times X$ the universal family of lines, there exists a cycle
$\Gamma_\Sigma\in {\rm CH}^2(J\times \Sigma)$, such that
\begin{eqnarray}\label{eqmshen}  P_{\Sigma*}\circ \Gamma_{\Sigma*}=\Gamma_{X*}:J\rightarrow J.
\end{eqnarray}
Note that by \cite{CG},
\begin{eqnarray}\label{eqisoajdroite} P_{\Sigma*}:{\rm Alb}(\Sigma)\rightarrow J^3(X)=J
\end{eqnarray}
is also an isomorphism.
In order to construct $\Gamma$, we apply the following Lemma \ref{lemmadroiteell} proved below
\begin{lemm} \label{lemmadroiteell} There exists a correspondence $\Gamma_{1,5}\in{\rm CH}^{10}(\Sigma\times \widetilde{H}_{5,1})$ inducing an isomorphism
$\Gamma_{1,5*}:{\rm Alb}(\Sigma)\rightarrow {\rm Alb}(\widetilde{H}_{5,1})$ compatible with the Abel-Jacobi isomorphisms
(\ref{eqisoajell}) and (\ref{eqisoajdroite}).
\end{lemm}
Indeed, we set $\Gamma:=\Gamma_{1,5}\circ \Gamma_\Sigma\in {\rm CH}^{10}(J\times \widetilde{H}_{5,1})$ and it follows from (\ref{eqmshen}) that $\Gamma$ is a universal $0$-cycle for $\widetilde{H}_{5,1}$.
\end{proof}
\begin{proof}[Proof of Lemma \ref{lemmadroiteell}] We choose a point $x\in X$ and a smooth  plane cubic curve $E$ passing through $x$. Let $\Delta\subset X$ be a general line. Let $Q$ be the plane $\langle \Delta, x\rangle$. Then $Q\cap X$ is the union of $\Delta$ and a conic $C$ passing through $x$. The union $C\cup E$ is a (reducible) elliptic curve of degree $5$, which is parameterized by  a smooth point of $H^{5,1}$, hence by a point of $\widetilde{H}_{5,1}$. This construction gives a rational map $\phi:\Sigma\dashrightarrow \widetilde{H}_{5,1}$. Let $\Gamma_{1,5}\in{\rm CH}^{10}(\Sigma\times \widetilde{H}_{5,1})$ be minus the graph of $\phi$. For any $\Delta$ as above, the classes in $X$ of the curves $\Delta$, $C$ and $ E$ satisfy the relations
$$ \Delta+C=h^2\,\,{\rm in}\,\,{\rm CH}^2(X),\,  E=h^2\,\,{\rm in}\,\,{\rm CH}^2(X),$$
hence the curve $C\cup E$ is rationally equivalent to $2h^2-\Delta$ in $X$ and we have
$$\Phi_X\circ \Gamma_{1,5*}=-\Phi_X\circ\phi_*=P_*:{\rm Alb}(\Sigma)\rightarrow J^3(X).$$
\end{proof}
\subsection{Hodge classes and  cycles classes on   Brauer-Severi varieties\label{secunibra}}
 Our goal is to establish  Theorem \ref{theobrauer0univ} concerning the existence of a universal $0$-cycle for  the total space $P\rightarrow A$ of a Brauer-Severi variety over an abelian variety $A$. We first start with the following easy result concerning the Hodge classes on the total space of a Brauer-Severi variety $p: P\rightarrow B$ of relative dimension $d$, where we assume that $B$ is smooth  projective and  $H^3(B,\mathbb{Z})$ has no torsion. In this case, the Brauer class $\alpha_P$ belongs to the $(d+1)$-torsion of  the group $$H^2(B,\mathcal{O}_B)/H^2(B,\mathbb{Z})\hookrightarrow  H^2(B,\mathcal{O}_{B}^*),$$
 where the sheaf $\mathcal{O}_{B}^*$  is the sheaf of invertible holomorphic functions on $B$ equipped with the Euclidean topology, and the inclusion above  is induced by the exponential exact sequence. The class $\alpha_P$  can be  constructed as follows. The   Brauer class measures the obstruction to the existence of an algebraic (or equivalently  holomorphic since  $B$ is projective) line bundle $L$ on $P$, whose restriction to the fibers $P_x\cong \mathbb{P}^d$ is the generator $\mathcal{O}(1)$. As there is no torsion in $H^3(B,\mathbb{Z})$, the Leray spectral sequence of $p$ shows that a topological such line bundle $H$ exists on $P$, and $c_1(H)\in H^2(P,\mathbb{Z})$ is well defined modulo $p^*H^2(B,\mathbb{Z})$.
We can choose $H$ to be holomorphic if we can arrange that $c_1(H)$ vanishes in $H^2(P,\mathcal{O}_P)$.
 The class $\alpha_P$ is defined as the image of $c_1(H)$ in $$H^2(P,\mathcal{O}_P)/p^*H^2(B,\mathbb{Z})\cong H^2(B,\mathcal{O}_B)/H^2(B,\mathbb{Z}).$$
This class is of $(d+1)$-torsion because $P$ carries a holomorphic line bundle whose restriction to the fibers $P_x\cong \mathbb{P}^d$ is the line bundle  $\mathcal{O}(d+1)$, namely the relative anticanonical bundle.
\begin{lemm} \label{lehdgbrauer} Let $\gamma$ be an integral  Hodge class of degree $2k$ on $B$. Then, if there exists an integral Hodge  class
$\tilde{\gamma}\in{\rm Hdg}^{2k+2d}(P,\mathbb{Z})$ such that $p_*\tilde{\gamma}=\gamma$, one has
\begin{eqnarray}\label{eqvanishingbra} \gamma\cup \alpha_P=0\,\,{\rm in}\,\,H^{2k+2}(B,\mathbb{C})/(F^{k+1}H^{2k+2}(B,\mathbb{C})+H^{2k+2}(B,\mathbb{Z})).\end{eqnarray}
\end{lemm}
In (\ref{eqvanishingbra}),  $F^\cdot$ denotes the Hodge filtration on the Betti cohomology of $B$ with complex coefficients and
 the cup-product $\gamma\cup \alpha_P$ is defined as follows. An integral Hodge class $\gamma$ of degree $2k$ on a smooth projective variety $Y$ can be seen as a pair
$(\gamma_\mathbb{Z},\gamma_F)$, with $\gamma_{\mathbb{Z}}\in H^{2k}(Y,\mathbb{Z})$, $\gamma_F\in F^kH^{2k}(Y,\mathbb{C})$ such that
\begin{eqnarray}\label{eqhodge} \gamma_{\mathbb{C}}=\gamma_F\,\,{\rm in}\,\,H^{2k}(Y,\mathbb{C}).\end{eqnarray}
Given such a Hodge class $\gamma$ on $Y$ and a Brauer class $$\alpha\in H^2(Y,\mathcal{O}_Y)/H^2(Y,\mathbb{Z})=H^2(Y,\mathbb{C})/(F^1H^2(Y,\mathbb{C})+H^2(Y,\mathbb{Z}))$$ with lift $\tilde{\alpha}\in H^2(Y,\mathcal{O}_Y)=H^2(Y,\mathbb{C})/F^1H^2(Y,\mathbb{C})$, we define
\begin{eqnarray}\label{eqcupbra} \gamma\cup \alpha:=\gamma_F\cup \tilde{\alpha}\in H^{2k+2}(Y,\mathbb{C})/F^{k+1}H^{2k+2}(Y,\mathbb{C})\,\,{\rm mod}\,\,H^{2k+2}(Y,\mathbb{Z}).\end{eqnarray}

\begin{proof}[Proof of Lemma \ref{lehdgbrauer}] We choose as before a lift
$\tilde{\alpha}_P$ of $\alpha_P$ in $H^2(Y,\mathcal{O}_Y)$.  By construction, the pull-back  $p^*\alpha_P$ vanishes in $H^2(P,\mathcal{O}_P)/H^2(P,\mathbb{Z})$. We thus have
\begin{eqnarray}\label{eq23novpourbr} p^*\tilde{\alpha}_P=\eta\,\,{\rm in}\,\, H^2(P,\mathbb{C}),\end{eqnarray}
where $\eta\in H^2(P,\mathbb{Z})$.
If $\tilde{\gamma}\in {\rm Hdg}^{2k+2d}(P,\mathbb{Z})$, with de Rham component $\tilde{\gamma}_F\in F^{k+d}H^{2k+2d}(P,\mathbb{C})$ and integral component $\tilde{\gamma}_\mathbb{Z}\in H^{2k+2d}(P,\mathbb{Z})$, it follows from (\ref{eqhodge}) and (\ref{eq23novpourbr}) that
\begin{eqnarray}\label{eq23novpourbrproof} \tilde{\gamma}_F\cup p^*\tilde{\alpha}_P =\tilde{\gamma}_\mathbb{Z}\cup \eta\,\,{\rm in}\,\,H^{2k+2+2d}(P,\mathbb{C})/F^{k+d+1}H^{2k+2+2d}(P,\mathbb{C}).
\end{eqnarray}
As the right hand side is an integral cohomology class (modulo torsion) on $P$, we conclude by push-forward to $B$ that
$p_*\tilde{\gamma}_{F}\cup \tilde{\alpha}_P=p_*(\tilde{\gamma}_{F}\cup p^*\tilde{\alpha}_P)$ is an integral cohomology class on $B$, modulo $F^{k+1}H^{2k+2}(B,\mathbb{C})$. Using the description (\ref{eqcupbra}) of the cup-product, we conclude that the Hodge class $\gamma=p_*\tilde{\gamma}$ satisfies
(\ref{eqvanishingbra}).
\end{proof}

\begin{proof}[Proof of Theorem \ref{theobrauer0univ}] Let $A$ be an abelian variety of dimension $g$ which is Mumford-Tate general with $\rho(A)=1$. The N\'{e}ron-Severi group ${\rm NS}(A)$ is $1$-dimensional, generated by the class $c_1(L)$ for some ample line bundle $L$ on $A$. We recall the notation $\gamma_{\rm min}$ for the generator of the cyclic group ${\rm Hdg}^{2g-2}(X,\mathbb{Z})$. Our  assumptions are that

(i)  the intersection number $\gamma_{\rm min}\cdot c_1(L)$ is even and

(ii) any Brauer-Severi variety $p:P\rightarrow A$ has a universal $0$-cycle.

We will choose a Brauer class $$\beta\in {\rm Tors}(H^2(A,\mathcal{O}_A)/H^2(A,\mathbb{Z}))\cong H^2(A,\mathbb{Z})_{\rm tr}\otimes \mathbb{Q}/\mathbb{Z}$$
with arbitrary divisible order,  where
$H^2(A,\mathbb{Z})_{\rm tr}:=H^2(A,\mathbb{Z})/{\rm NS}(A)$. Let $p:P_\beta\rightarrow A$ be a Brauer-Severi variety on $A$, of Brauer class $\beta$. Let $d_\beta$ be the  dimension of $P_\beta$.
 As, by assumption (ii), $P_\beta$ has a universal zero-cycle for $P_\beta$, there  is a cycle $Z_\beta$ in ${\rm CH}^{d_\beta}(A\times P_\beta)$, such that the class $[W_\beta]:=(Id,p)_*[Z_\beta]\in {\rm Hdg}^{2g}(A\times A,\mathbb{Z})$ has $(1,1)$-K\"{u}nneth component equal to $\delta_1$, that is, acts as the identity on $H_1(A,\mathbb{Z})$.

We will prove the following
\begin{prop}\label{lecomputclassbrauer} Assume $\beta $ comes from a general class in $H^2(A,\mathbb{Q})_{\rm tr}$ with arbitrarily divisible denominator. Then the Hodge class $[W_\beta]$ on $A\times A$ satisfies
\begin{eqnarray}\label{eqpourWbeta} [W_\beta]=\alpha_0\delta_0+\delta_1+\alpha_2 c_1(L)\otimes \gamma_{\rm min}+ \gamma_\beta\,\,{\rm in}\,\,{\rm Hdg}^{2g}(A\times A,\mathbb{Z}),
\end{eqnarray}
where $\alpha_0$ and $\alpha_2$ are integral and $\gamma_\beta\in {\rm Hdg}^{2g}(A\times A,\mathbb{Z})$ is arbitrarily divisible.
\end{prop}
Assuming the proposition, the proof of Theorem \ref{theobrauer0univ} is concluded as follows. As $A$ is very general, the rational Hodge conjecture is known for $A\times A$ (see the proof of Lemma \ref{leeqmubeta}). It follows that any integral Hodge class which is  sufficiently divisible is algebraic on $A\times A$ and in particular the Hodge classes $\gamma_\beta$ of (\ref{eqpourWbeta}) with arbitrarily high divisibility are algebraic on $A\times A$. The class $\alpha_0\delta_0=\alpha_0[A\times {\rm pt}]$ is algebraic, and
$[W_\beta]=[(Id,p)_*Z_\beta]$ is also algebraic, so we conclude from (\ref{eqpourWbeta}) that the class
\begin{eqnarray}\label{eqZdelta} \delta_1+ \alpha_2 c_1(L)\otimes \gamma_{\rm min}\end{eqnarray}
is algebraic.

It remains to prove that  this implies  that $\delta_1$ is algebraic on $A\times A$. If we let act on $A\times A$ the endomorphism
$\mu'_-:=(Id_A,\mu_-)$, where $\mu_-$ is the multiplication by $-1$ on $A$, we deduce from the algebraicity of (\ref{eqZdelta}) that the class
$-\delta_1+ \alpha_2 c_1(L)\otimes \gamma_{\rm min}$ is also algebraic on $A\times A$, so that $2\alpha_2 c_1(L)\otimes \gamma_{\rm min}$ is algebraic on $A\times A$. If we now take the square (in the sense of the composition of correspondences) of
(\ref{eqZdelta}), we find that the class
$$\delta_1+\alpha_2^2 (c_1(L)\cdot\gamma_{\rm min})c_1(L)\otimes \gamma_{\rm min}$$
is also algebraic. By assumption (i), $c_1(L)\cdot\gamma_{\rm min}$ is even, so we conclude that the class $\alpha_2^2 (c_1(L)\cdot\gamma_{\rm min})c_1(L)\otimes \gamma_{\rm min}$ is algebraic, and finally $\delta_1$ is algebraic, so $A$ is split.
\end{proof}
For the proof of Proposition \ref{lecomputclassbrauer}, we will use the following Lemma \ref{leeqmubeta}. Hodge classes of degree $2g$ on $A\times A$ decompose according to the K\"unneth decomposition
$$\gamma=\sum_i\gamma_i,$$
where $\gamma_i=\gamma\circ \delta_i\in H^i(A,\mathbb{Z})\otimes H^{2g-i}(A,\mathbb{Z})$. For any class $\tilde{\beta}\in H^2(A,\mathbb{Q})$, the cup-product by  ${\rm pr}_2^*\tilde{\beta}$  preserves the K\"unneth  decomposition and induces a morphism
\begin{eqnarray}\label{eqmubeta} \overline{{\rm pr}_2^*\tilde{\beta}\cup} : {\rm Hdg}^{2g}(A\times A,\mathbb{Q})\rightarrow H^{2g+2}(A\times A,\mathbb{Q})/{\rm Hdg}^{2g+2}(A\times A,\mathbb{Q}).
\end{eqnarray}

\begin{lemm} \label{leeqmubeta} For a Mumford-Tate  general abelian variety $A$ with Picard number $1$ and  polarizing class $l$, and a generic class $\tilde{\beta}\in H^2(A,\mathbb{Q})$, the morphism
$ \overline{\tilde{\beta}\cup}$ of (\ref{eqmubeta}) has for kernel the $\mathbb{Q}$-vector subspace generated by
\begin{eqnarray}\label{eqthrreclasses} \delta_0\in H^{0}(A,\mathbb{Q})\otimes H^{2g}(A,\mathbb{Q}),\,\delta_1\in H^{1}(A,\mathbb{Q})\otimes H^{2g-1}(A,\mathbb{Q}),\\
\nonumber l\otimes\gamma_{\rm min} \in H^{2}(A,\mathbb{Q})\otimes H^{2g-2}(A,\mathbb{Q}).\end{eqnarray}
\end{lemm}
\begin{proof} As the cup-product by ${\rm pr}_2^*\tilde{\beta}$ acts as $Id\otimes (\tilde{\beta}\cup)$ on
$H^{i}(A,\mathbb{Q})\otimes H^{2g-i}(A,\mathbb{Q})$,  it is clear for degree or Hodge type reasons that the three classes $\delta_0$, $\delta_1$ and $l\otimes\gamma_{\rm min}$ belong to the kernel of $ \overline{\tilde{\beta}\cup}$. In the other direction,
let us   describe the rational Hodge classes on $A\times A$. By assumption, the  Mumford-Tate group of $A$ is the symplectic group ${\rm Sp}(2g)$, hence the Lefschetz decomposition
\begin{eqnarray}\label{eqlefdec} H^i(X,\mathbb{Q})=\oplus_{i-2j\geq 0} l^j\cup H^{i-2j}(A,\mathbb{Q})_{\rm prim}\end{eqnarray}
for $i\leq g$, is a decomposition into simple Hodge structures and there are no nonzero morphisms of Hodge structures
$$H^{i-2j}(A,\mathbb{Q})_{\rm prim}\rightarrow H^{i-2j'}(A,\mathbb{Q})_{\rm prim}$$
for $j\not=j'$. It follows that the rational Hodge classes in
$$H^{2g-i}(A,\mathbb{Q})\otimes H^i(A,\mathbb{Q})={\rm End}(H^i(A,\mathbb{Q})),$$
namely the morphisms of Hodge structures in ${\rm End}(H^i(A,\mathbb{Q}))$, are linear combinations of the projectors $\pi_{i,j}$ of the Lefschetz decompositions (\ref{eqlefdec}). It is well-known that these projectors are algebraic with rational coefficients (this is a consequence of the Lefschetz standard conjecture for abelian varieties, see \cite{lieberman}). When $i\geq g$, the discussion is the same, except that we use the Lefschetz isomorphism
$$H^{i}(X,\mathbb{Q})\cong H^{2g-i}(X,\mathbb{Q})$$
and the Lefschetz decomposition on $H^{2g-i}(X,\mathbb{Q})$.
We now observe that, for any $i$, and for any Hodge class $\phi$  in the subspace  $${\rm Hdg}^{2g}(A\times A,\mathbb{Q})\cap {\rm End}(H^{i}(A,\mathbb{Q}))\subset {\rm End}_0(H^*(A,\mathbb{Q}))=H^{2g}(A\times A,\mathbb{Q})$$ generated by all
the $\pi_{i,j}$ except  the three classes appearing in (\ref{eqthrreclasses}), the  image ${\rm Im}\,\phi\subset H^{i}(A,\mathbb{Q})$ is a Hodge structure with a nonzero component  $({\rm Im}\,\phi)^{p,q}$ for some  $q\leq g-2$. Furthermore, there are only finitely many such Hodge substructures, since they all must be  direct sums of Lefschetz components appearing in the Lefschetz decomposition (\ref{eqlefdec}). One then easily checks that
for a generic $\eta\in H^2(A,\mathcal{O}_A)$, and any $\phi$ as above, the cup-product map
$$\eta:({\rm Im}\,\phi)^{p,q}\rightarrow H^{p,q+2}(A)$$
is nonzero, hence $\eta\cup \phi\not=0$ in ${\rm Hom}(H^{p,q}(A),H^{p,q+2}(A))$.
It follows from the above discussion that
for a general $\eta\in H^2(A,\mathcal{O}_A)$, the cup-product map
$$\eta\cup: {\rm Hdg}^{2g}(A\times A,\mathbb{Q})\rightarrow H^{2g+2}(A\times A,\mathbb{C})/F^{g+1}H^{2g+2}(A\times A,\mathbb{C})$$
has for kernel the space generated by (\ref{eqthrreclasses}). This  implies the lemma because the image of $H^2(A,\mathbb{Q})$ in $H^2(A,\mathcal{O}_A)$ is Zariski dense and one has the following commutative diagram for any  $\tilde{\beta}\in H^2(A,\mathbb{Q})$ with image $\eta \in H^2(A,\mathcal{O}_A)$
$$  \begin{matrix}
 & {\rm Hdg}^{2g}(A\times A,\mathbb{Q})&\stackrel{\overline{\tilde{\beta}\cup}}{\rightarrow}& H^{2g+2}(A\times A,\mathbb{Q})/{\rm Hdg}^{2g+2}(A\times A,\mathbb{Q})
 \\
&\parallel& & \downarrow
\\
&{\rm Hdg}^{2g}(A\times A,\mathbb{Q})&\stackrel{\eta\cup }{\rightarrow} &H^{2g+2}(A,\mathbb{C})/F^{g+1}H^{2g+2}(A,\mathbb{C}).
\end{matrix}
$$
\end{proof}
\begin{proof}[Proof of Proposition \ref{lecomputclassbrauer}]
Lemma \ref{lehdgbrauer} applied to the Brauer-Severi variety $A\times P_\beta\rightarrow A\times A$   says that
\begin{eqnarray}\label{eqcontbrauer} [W_\beta]\cup {\rm pr}_2^*\beta=0\,\,{\rm in}\,\,{\rm Tors}(H^{2g+2}(A\times A,\mathbb{C})/(F^{g+1}H^{2g+2}(A\times A,\mathbb{C})+H^{2g+2}(A\times A,\mathbb{Z})))\\
\label{eqsecondeequality}
=H^{2g+2}(A\times A,\mathbb{Q})/({\rm Hdg}^{2g+2}(A\times A,\mathbb{Q})+H^{2g+2}(A\times A,\mathbb{Z})).
\end{eqnarray}
The second equality (\ref{eqsecondeequality}) follows from the fact that $${\rm Hdg}^{2g+2}(A\times A,\mathbb{Q})={\rm Ker}\,( H^{2g+2}(A\times A,\mathbb{Q})\rightarrow   H^{2g+2}(A\times A,\mathbb{C})/F^{g+1}H^{2g+2}(A\times A,\mathbb{C})).$$
Equation (\ref{eqcontbrauer}) says equivalently that, with the notation of (\ref{eqmubeta}),
\begin{eqnarray} \label{eqvanmodZ} \overline{{\rm pr}_2^*\tilde{\beta}\cup} [W_\beta]=0\,\,{\rm in} \, \,\,(H^{2g+2}(A\times A,\mathbb{Z})/{\rm Hdg}^{2g+2}(A\times A,\mathbb{Z}))\otimes \mathbb{Q}/\mathbb{Z}.
\end{eqnarray}
We now apply the following elementary
\begin{lemm} \label{lelatticeelem} Let $H_1,\,H_2$ be two lattices, and $\psi \in  {\rm Hom}\,(H_1,H_2)\otimes \mathbb{Q}$ be an injective morphism. Then there exists an integer $d$, such that, for any integer $N$, and any $h\in H_1$ with
$\frac{1}{N}\psi(h)\in H_2\subset H_2\otimes \mathbb{Q}$, one has
$h\in \frac{N}{d} H_1$. In particular, $h\in H_1$ if $d$ divides $N$, and $h\in H_1$ is arbitrarily divisible if $\frac{N}{d}$ is.
\end{lemm}
We apply this lemma to $$H_1={\rm Hdg}^{2g}(A\times A,\mathbb{Z})/\langle\delta_0,\,\delta_1,\,l\otimes \gamma_{\rm min}\rangle,\,\,  H_2={ H}^{2g+2}(A\times A,\mathbb{Z})/{\rm Hdg}^{2g+2}(A\times A,\mathbb{Z}),$$ taking for  $\psi$ the cup-product map  $\overline{{\rm pr}_2^*\tilde{\beta}\cup}$. It satisfies our assumptions by Lemma \ref{leeqmubeta}. We thus conclude that (\ref{eqvanmodZ}) implies that, when $N$ is arbitrarily divisible,  the class
$[W_{\frac{1}{N}\beta}]$ is arbitrarily divisible modulo $\langle\delta_0,\,\delta_1,\,l\otimes \gamma_{\rm min}\rangle$. As we know furthermore that the K\"{u}nneth component $[W_{\frac{1}{N}\beta}]$ is equal to $\delta_1$, this concludes the proof of Proposition  \ref{lecomputclassbrauer}.
\end{proof}

\subsection{On the existence of a universal codimension $2$ cycle for a cubic threefold  \label{sectheocubicsplit}}
 This section is devoted to the proof of Theorem \ref{theocubicsplit}. Let $X$ be a cubic threefold and $J=J^3(X)$ its intermediate Jacobian. This is a 5-dimensional principally polarized abelian variety.
 We assume  that $X$ admits a   universal codimension $2$ cycle  $\Gamma\in{\rm CH}^2(J\times X)$ and we want to prove, under the  assumption that $X$ is Mumford-Tate general,  that $J$ is split.  Note that, especially in view of Proposition \ref{procubiceasy}, the statement has some  similarities with Theorem \ref{theobrauer0univ}. There are however two differences. First of all, in the cubic case, the Iliev-Markushevich-Tikhomirov construction does not provide  a Brauer-Severi variety but only a Brauer-Severi variety over a Zariski open set of $J$. Secondly, in the cubic case, we are given  only one  (generic) Brauer-Severi variety admitting a universal $0$-cycle, while in Theorem \ref{theobrauer0univ}, we are given Brauer-Severi varieties admiting a universal $0$-cycle, with  Brauer class  general of arbitrarily high order.

 Let $\Sigma=F_1(X)$ be the surface of lines in $X$. We will first prove some preparatory lemmas.
 By \cite{CG},  the universal line
$$P\subset \Sigma\times X,$$
 induces an embedding $j=\Phi_X\circ P_*:\Sigma\rightarrow J^3(X)=J$
and an isomorphism
\begin{eqnarray}\label{eqphiP}P_*=j_*:{\rm Alb}(\Sigma)\rightarrow J.\end{eqnarray}
According to \cite{mshen}, given $\Gamma$, there exists a correspondence $\Gamma'\in{\rm CH}^2(J\times \Sigma)$ such that
\begin{eqnarray}\label{eqPgamamgammaprime0} P_*\circ \Gamma'_*=\Gamma_*: {\rm CH}_0(J)_{\rm hom}\rightarrow {\rm CH}^2(X)_{\rm hom}.\end{eqnarray}
As ${\rm CH}^2(X)_{\rm hom}\cong J^3(X)=J$ via the Abel-Jacobi map $\Phi_X$ (see \cite{CG}, and \cite{blochsrinivas} for a  more general result), (\ref{eqPgamamgammaprime0}) is equivalent to the fact that
\begin{eqnarray}\label{eqPgamamgammaprime1} P_*\circ \Gamma'_*=\Gamma_*=Id_J: J\rightarrow J.\end{eqnarray}
Via the isomorphism (\ref{eqphiP}), we can write (\ref{eqPgamamgammaprime1}) as

\begin{eqnarray}\label{eqPgamamgammaprime} j_*\circ \Gamma'_*=Id_J,\end{eqnarray}
and equivalently, looking at the action of these correspondences on homology
\begin{eqnarray}\label{eqPgamamgammaprimecoh} j_*\circ [\Gamma']_*=Id_{H_1(J,\mathbb{Z})}.\end{eqnarray}
Our strategy will be   to modify $\Gamma'$ by composing it with self-correspondences of $\Sigma$ so as to achieve the condition
\begin{eqnarray} \label{eqclassprojkun}[j\circ \Gamma']=\delta_1,\end{eqnarray} where $\delta_1$ is the K\"{u}nneth projector on $H_1(J,\mathbb{Z})$.
We first note the following
\begin{lemm} \label{lekunprojsigma} Denote by  $\delta_{1,\Sigma}\in H^4(\Sigma\times \Sigma,\mathbb{Z})$ the K\"unneth projector onto $H_1(\Sigma,\mathbb{Z})$. Then $2\delta_{1,\Sigma}$ is algebraic.
Equivalently,  twice the K\"unneth projector $\delta_{3,\Sigma}$ onto  $H_3(\Sigma,\mathbb{Z})$ is algebraic.
\end{lemm}
\begin{rema}{\rm The integral cohomology of $\Sigma$ is torsion free (see \cite{collino}), so the $\delta_{i,\Sigma}$ are well defined. }
\end{rema}
\begin{coro}\label{corotardifquoiqueutile} Twice the K\"unneth projector $\delta_{2,\Sigma}$ of $\Sigma$ is also algebraic.
\end{coro}
\begin{proof} Indeed we have $2\delta_{2,\Sigma}=2[\Delta_\Sigma]-2\delta_{1,\Sigma}-2\delta_{3,\Sigma}-2\delta_{0,\Sigma}-2\delta_{4,\Sigma}$, where $2\delta_{1,\Sigma}$ and $2\delta_{3,\Sigma}$ are algebraic by Lemma \ref{lekunprojsigma}, and $\delta_{0,\Sigma}=[\Sigma\times {\rm pt}]$, $\delta_{4,\Sigma}=[{\rm pt}\times \Sigma ]$ are also algebraic, as already mentioned.
\end{proof}
\begin{proof}[Proof of Lemma \ref{lekunprojsigma}] We know by \cite{CG}  that the class $l\in H^2(\Sigma,\mathbb{Z})$ of the curve $C_\Delta\subset \Sigma$ of lines  meeting a given line $\Delta\subset X$ satisfies
\begin{eqnarray} \label{eqajoutedefl}2l=j^*\theta,\end{eqnarray}
where $\theta\in H^2(J,\mathbb{Z})$ is the class of  a Theta-divisor.
Let $(j,j)^*\mathcal{P}$ be the pull-back to $\Sigma\times \Sigma$ of a Poincar\'{e} divisor $\mathcal{P}$ on $J\times J$, so that
$$[(j,j)^*\mathcal{P}]\in H^1(\Sigma,\mathbb{Z})\otimes H^1(\Sigma,\mathbb{Z})\subset H^2(\Sigma,\mathbb{Z})$$
is algebraic. The class
\begin{eqnarray}\label{eqactioncalP}\gamma:=[(j,j)^*\mathcal{P}]\cup {\rm pr_2}^*l\in H^1(\Sigma,\mathbb{Z})\otimes H^3(\Sigma,\mathbb{Z})\subset H^4(\Sigma,\mathbb{Z})\end{eqnarray}
is thus algebraic. It is clear from (\ref{eqactioncalP}) that $\gamma$ acts trivially on $H_i(\Sigma,\mathbb{Z})$ for $i\not=1$. The action of
$\gamma$ on $H_1(\Sigma,\mathbb{Z})$ is given by
\begin{eqnarray} \label{eqactionprojsigma} \gamma_*(u)=l\cup j^*([\mathcal{P}]_*(j_*u))\,\,{\rm in}\,\,H^3(\Sigma,\mathbb{Z})\cong H_1(\Sigma,\mathbb{Z}).
\end{eqnarray}
It remains to see that the right hand side is equal to $2u$. Pushing forward to $J$, we get, using the fact that $l=\frac{1}{2}j^*\theta$,
$$j_*(\gamma_*(u))=\frac{1}{2}\theta\cup j_*j^* ([\mathcal{P}]_*(j_*u))=\frac{1}{2}\theta\cup [\Sigma]\cup  [\mathcal{P}]_*(j_*u)\,\,{\rm in}\,\,H^{9}(J,\mathbb{Z})\cong H_1(J,\mathbb{Z}).$$
The right hand side is equal to $2j_*u$ because
$$[\Sigma]=\frac{\theta^3}{3!},\,\,\gamma_{\rm min}=\frac{\theta^4}{4!},$$
where the first equality is proved in \cite{CG}, and
$[\mathcal{P}]_*:H_1(J,\mathbb{Z})\rightarrow H^1(J,\mathbb{Z})$ is the inverse of $\gamma_{\rm min} \cup : H^1(J,\mathbb{Z})\rightarrow H^9(J,\mathbb{Z})=H_1(A,\mathbb{Z})$.
\end{proof}
We  now study   the action of $j_*\circ [\Gamma']_*$ on the other homology groups of $J$.
\begin{lemm} \label{leh32} The image of $j_*\circ [\Gamma']_*:H_3(J,\mathbb{Z})\rightarrow H_3(J,\mathbb{Z})$ is contained in
$2j_*H_3(\Sigma,\mathbb{Z})$.
\end{lemm}
\begin{proof} We observe that, with $\mathbb{Q}$-coefficients, the Hodge structure on $H_3(J,\mathbb{Q})=H^7(J,\mathbb{Q})$ splits by the Lefschetz decomposition as
\begin{eqnarray}\label{eqdech3A} H^7(J,\mathbb{Q})=\theta^2\cup H^3(J,\mathbb{Q})_{\rm prim}\oplus \theta^3\cup H^1(J,\mathbb{Q}),\end{eqnarray}
where the  Hodge structures of the two summands are simple, have  no nontrivial endomorphisms and admit no non-trivial morphisms from one to the other. These three facts follow from the fact that, by assumption,  the Mumford-Tate group of the   Hodge structure on $H^1(J,\mathbb{Q})$ is the symplectic group of $(H^1(J,\mathbb{Q}),\langle\,,\,\rangle_\theta)$. Note that, as $[\Sigma]=\frac{\theta^3}{3!}$ and $j^*:H^1(J,\mathbb{Q})\rightarrow H^1(\Sigma,\mathbb{Q})$ is surjective, the space
$\theta^3H^1(J,\mathbb{Q})$ is equal to $j_*H^1(\Sigma,\mathbb{Q})$.
As $j_*\circ [\Gamma']_*$ is a morphism of Hodge structures and the image of $j_*\circ [\Gamma']_*$ is contained in ${\rm Im}\,j_*$, it follows from the decomposition (\ref{eqdech3A}) with its stated properties that $j_*\circ [\Gamma']_*$ is a multiple $\lambda \pi$ of the projector $\pi$ on $j_*H^1(\Sigma,\mathbb{Q})=\theta^3\cup H^1(J,\mathbb{Q})$ associated with the decomposition (\ref{eqdech3A}). Our statement is thus that the coefficient $\lambda$ is an even integer. Using the fact that $[\Sigma]=\frac{\theta^3}{3!}$, one  easily  shows that the sublattice
$$j_*H^1(\Sigma,\mathbb{Z})=[\Sigma]\cup H^1(J,\mathbb{Z})\subset H^7(J,\mathbb{Z})$$
is primitive. As $\lambda \pi=j_*\circ[\Gamma']_*: H^7(J,\mathbb{Z})\rightarrow H^7(J,\mathbb{Z})$ is equal to $\lambda Id$ on $[\Sigma]\cup H^1(J,\mathbb{Z})$, it follows that $\lambda$ is an integer.
To see that $\lambda $ must be even, we observe that $H^7(J,\mathbb{Z})$ carries a unimodular intersection pairing $\omega_7$ thanks to the principal polarization of $J$, which provides isomorphisms $H^i(J,\mathbb{Z})\cong H_{i}(J,\mathbb{Z})$ for all $i$. Passing to $\mathbb{Z}$-coefficients, the decomposition (\ref{eqdech3A}) provides
the inclusion of a finite index sublattice
 \begin{eqnarray}\label{eqdech3AZ} = \frac{\theta^3}{3!}\cup H^1(J,\mathbb{Z})\oplus (\frac{\theta^3}{3!}\cup H^1(J,\mathbb{Z}))^{\perp_{\omega_7}}\subset H^7(J,\mathbb{Z}),\end{eqnarray}
 where the orthogonal decomposition induces (\ref{eqdech3A}) after  passing to rational coefficients.
 If there exists an integral endomorphism of the lattice $H^7(J,\mathbb{Z})$ which acts as an odd multiple $\lambda \pi$ of the orthogonal projector from $H^7(J,\mathbb{Z})$ onto the first summand in (\ref{eqdech3AZ}), then the discriminant of
the restriction of the pairing $\omega_7$ to $\frac{\theta^3}{3!}\cup H^1(J,\mathbb{Z})$ is odd. But, by Lemma \ref{lepolcomp} proved below, and using as above the isomorphism
$$H^7(J,\mathbb{Z})\cong H_7(J,\mathbb{Z})\cong H^3(J,\mathbb{Z})$$
given by the principal polarization and Poincar\'{e} duality,  this restriction  is equal to  four times the unimodular pairing on $H^1(J,\mathbb{Z})$, which is a contradiction.
\end{proof}
\begin{lemm}\label{lepolcomp} Let $A$ be a principally polarized abelian variety of dimension $5$ and $\theta$ the class of its Theta divisor. Then the unimodular pairing on $H^3(A,\mathbb{Z})$ restricts to $4$ times the theta pairing on $H^1(A,\mathbb{Z})\cong \theta\cup H^1(A,\mathbb{Z})\subset H^3(A,\mathbb{Z})$.
\end{lemm}
\begin{proof}  The  unimodular pairing on $H^3(A,\mathbb{Z})$ is given by the composite isomorphism
\begin{eqnarray} \label{eqisopolA3} \iota :  H^3(A,\mathbb{Z})\cong\bigwedge^3H^1(A,\mathbb{Z})\cong\bigwedge^3H_1(A,\mathbb{Z})\cong H_3(A,\mathbb{Z}),
\end{eqnarray}
where the middle isomorphism is induced by $\theta$, the left isomorphism is given by cup-product and the last isomorphism is given by Pontryagin product.
This isomorphism maps
$\theta \cup H^1(A,\mathbb{Z})$ to $\gamma_{\rm min}* H_1(A,\mathbb{Z})$. We thus have to compute, for $\alpha\in H^1(A,\mathbb{Z}),\,\beta\in H_1(A,\mathbb{Z})$ the pairing $\langle \theta\cup \alpha,\gamma_{\rm min}*\beta\rangle$. The statement is that it is equal to $4$ times the pairing $\langle  \alpha,\beta\rangle$.
The statement is topological so we can assume that $A=JC$ is the Jacobian of a curve $C\subset A$ whose class $[C]\in H_2(A,\mathbb{Z})$ is the minimal class $\gamma_{\rm min}$.
Let $\mu:C\times C\rightarrow A$ be the sum map. Let $\beta$ be represented by the class of an oriented circle $B\subset C$, and let
$$\mu_B:C\times B\rightarrow A$$
be the restriction of $\mu$ to $C\times B$. We have by definition
$$\gamma_{\rm min}*\beta=\mu_*([C\times B]_{\rm fund})\,\,{\rm in}\,\,H_3(A,\mathbb{Z}),$$
so that
\begin{eqnarray}\label{eqnombredint} \langle \theta\cup \alpha,\gamma_{\rm min}*\beta\rangle=\int_{C\times B} \mu_B^*(\theta\cup \alpha).
\end{eqnarray}
We observe that $\mu_B^*\theta={\rm pr}_1^*\theta_C+\mathcal{P}$, where $\mathcal{P}\in H^1(C,\mathbb{Z})\otimes H^1(B,\mathbb{Z})$ is now the pull-back to $C\times B$ of the Poincar\'{e} divisor of $J\times J$ and ${\rm pr}_1$ is the first projection $C\times B\rightarrow C$. Furthermore
$$\mu_B^*\alpha={\rm pr}_1^*\alpha_{\mid C}+{\rm pr}_2^*\alpha_{\mid B}.$$
Furthermore, the degree of $\theta_C$ is equal to $5$. It follows that
\begin{eqnarray}\label{eqnombredint2}\int_{C\times B} \mu_B^*(\theta\cup \alpha)=5\int_{{\rm pt}\times B}\alpha_{\mid B}+\int_{C\times B}[\mathcal{P}]\cup {\rm pr}_1^*\alpha_{\mid C}.
\end{eqnarray}
The first term is equal to $5\langle\alpha,\beta\rangle$.
By definition of the Poincar\'{e} divisor, the second term equals $\int_C\alpha_C\cup [\mathcal{P}]^*\beta=-\langle\alpha,\beta\rangle$, hence (\ref{eqnombredint}) and (\ref{eqnombredint2}) give
$$\langle \theta\cup \alpha,\gamma_{\rm min}*\beta\rangle=4\langle\alpha,\beta\rangle.$$
\end{proof}

\begin{proof}[Proof of Theorem \ref{theocubicsplit}]
We want to prove that the existence of a universal codimension $2$-cycle  $\Gamma$ implies that $J$ is split, that is, the K\"{u}nneth projector $\delta_1$ onto $H_1(J,\mathbb{Z})$ is algebraic. Note the following
\begin{lemm} Assuming the existence of a universal codimension $2$-cycle  $\Gamma$,  the algebraicity on $J\times J$ of the K\"{u}nneth projector $\delta_1$ onto $H_1(J,\mathbb{Z})$ is implied by the algebraicity on $\Sigma\times \Sigma$ of the K\"{u}nneth projector $\delta_{1,\Sigma}$ onto $H_1(\Sigma,\mathbb{Z})$.
\end{lemm}
\begin{proof}
Indeed, given $\Gamma$, let $\Gamma'\in{\rm CH}^2(X\times \Sigma)$ be a Shen cycle,  satisfying the equivalent conditions  (\ref{eqPgamamgammaprime}), (\ref{eqPgamamgammaprimecoh}). Let
$\Delta_{1,\Sigma}$ be a codimension $2$ cycle on $\Sigma$ such that $[\Delta_{1,\Sigma}]=\delta_{1,\Sigma}$. Let
$$\Delta_1:=j_*\circ \Delta_{1,\Sigma}\circ \Gamma'\in{\rm CH}^5(J\times J).$$
Then $[\Delta_{1}]_*=0$ on $H_i(J,\mathbb{Z})$ for $i\not=1$, since
\begin{eqnarray} \label{eqcorrespocomposee} [\Delta_{1}]_*=j_*\circ[\Delta_{1,\Sigma}]_*\circ [\Gamma']_*:H_*(J,\mathbb{Z})\rightarrow H_*(J,\mathbb{Z})\end{eqnarray}
and   $[\Delta_{1,\Sigma}]_*=0$ on $H_i(\Sigma,\mathbb{Z})$ for $i\not=1$. Furthermore, by (\ref{eqcorrespocomposee}) and (\ref{eqPgamamgammaprimecoh}), $[\Delta_{1}]_*$ acts as the identity on $H_1(J,\mathbb{Z})$.
\end{proof}
Note also that, by Lemma \ref{lekunprojsigma}, in order to prove that the K\"unneth projector $\delta_{1,\Sigma}$ is algebraic, it suffices to prove that an odd multiple $\lambda \delta_{1,\Sigma}$ is algebraic, or equivalently an odd multiple $\lambda \delta_{3,\Sigma}$ is algebraic.

We now consider the K\"unneth components $[\Gamma'']_i$ of the class of the cycle
$$\Gamma'':=\Gamma'\circ j\in {\rm CH}^2(\Sigma\times \Sigma).$$
We know by (\ref{eqPgamamgammaprimecoh}) that $[\Gamma'']_1$ acts as the identity on $H_1(\Sigma,\mathbb{Z})$, and, by Lemma \ref{leh32}, that $[\Gamma'']_3$ acts by an even multiple $2\lambda_3$ of the identity on $H_3(\Sigma,\mathbb{Z})$. It follows from Lemma \ref{lekunprojsigma} that there exists an algebraic cycle $\Gamma''_3$ on $\Sigma\times \Sigma$ acting as $2\lambda_3Id$ on $H_3(\Sigma,\mathbb{Z})$ and by $0$ on the other homology groups of $\Sigma$. Hence $$\Gamma''':=\Gamma''-\Gamma''_3$$
 acts by $0$ on $H_3(\Sigma,\mathbb{Z})$ and by $Id$ on $H_1(\Sigma,\mathbb{Z})$. We can in an obvious way also modify $\Gamma'''$ so that it acts trivially on $H_0(\Sigma,\mathbb{Z})$ and $H_4(\Sigma,\mathbb{Z})$.

Finally, we have to consider what happens on $H_2(\Sigma,\mathbb{Z})$. In fact, $H_2(\Sigma,\mathbb{Z})$ contains the finite index sublattice
\begin{eqnarray}\label{eqsublatticesubm} \mathbb{Z}l\oplus   (\mathbb{Z}l)^{\perp},\end{eqnarray}
where the class $l$ is defined in (\ref{eqajoutedefl}) and satisfies $l^2=5$, implying that the index of the sublattice (\ref{eqsublatticesubm}) is odd. Furthermore, the decomposition in (\ref{eqsublatticesubm}) is a direct sum of Hodge structures, the first one being trivial, and the second one being simple and nontrivial, because $X$ is Mumford-Tate general. It follows that the K\"unneth component 
$[\Gamma'']_2$ acts on $H_2(\Sigma,\mathbb{Z})$  preserving the sublattice (\ref{eqsublatticesubm}) and its  decomposition, hence by multiplication by respective  integers $\lambda_1$, $\lambda_2$ on the  summands.
The fact that $l^2=5$ shows that the cycle   $l\times l$ on $\Sigma$ acts by multiplication by $5$ on the first summand, so that the cycle
$$5\Gamma'''-\lambda_1 l\times l$$
has the property that its cohomology class $[5\Gamma'''-\lambda_1 l\times l]$
acts by $0$ on $l$, and by $5\lambda_2$ on $(\mathbb{Z}l)^{\perp}$.

We finally discuss the parity of $\lambda_2$.

\vspace{0.5cm}

{\bf Case (i).} {\it  $5\lambda_2=2m$ is even}. We know  by Corollary \ref{corotardifquoiqueutile} that $2\delta_2$ is algebraic, hence there exists a codimension $2$ cycle $\Delta_2$ on $\Sigma\times \Sigma$ such that $[\Delta_2]_*$ acts as $2m Id$ on
$(\mathbb{Z}l)^{\perp}$, and $0$ on $\mathbb{Z}l$ and the other homology groups of $\Sigma$. But then the cycle
$$5\Gamma'''-\Delta_2-\lambda_1 l\times l$$ acts by $0$ on $H_i(\Sigma,\mathbb{Z})$ for $i\not=0$ and by an odd multiple of the identity on $H_1(\Sigma,\mathbb{Z})$, which concludes the proof in this case.

\vspace{0.5cm}

{\bf Case (ii).} {\it  $5\lambda_2=2m+1$ is odd.} In this case, $(2m+1)\Delta_\Sigma-5\Gamma'''$ acts by $0$ on $(\mathbb{Z}l)^{\perp}$, by $(2m+1-5)Id$ on $H_1(\Sigma,\mathbb{Z})$, and by $(2m+1)Id$ on $H_3(\Sigma,\mathbb{Z})$. As $2m+1-5=2k$ is even, there exists by Lemma \ref{lekunprojsigma} a codimension $2$ cycle $\Delta_1$ on $\Sigma\times\Sigma$
such that $[\Delta_1]=2k\delta_1$ and then the cycle $$(2m+1)\Delta_\Sigma-5\Gamma'''-\Delta_1$$ acts by $0$ on $(\mathbb{Z}l)^{\perp}$ and  on $H_1(\Sigma,\mathbb{Z})$, and by $(2m+1)Id$ on $H_3(\Sigma,\mathbb{Z})$. Its class is thus an odd multiple of $\delta_3$. An odd multiple of $\delta_3$ is thus algebraic, hence an odd multiple of $\delta_1$ is algebraic.
\end{proof}

    \end{document}